\documentclass[11pt]{amsart}
\usepackage{amscd}
\usepackage{txfonts,mathptmx}
\usepackage{mathrsfs}
\usepackage[msc-links]{amsrefs}

\numberwithin{equation}{section}

\newtheorem{thm}{Theorem}[section]
\newtheorem{lem}[thm]{Lemma}
\newtheorem{pro}[thm]{Proposition}
\newtheorem{cor}[thm]{Corollary}
\newtheorem{claim}{Claim}

\theoremstyle{definition}

\theoremstyle{remark}

\newtheorem*{notation}{Notation}
\newtheorem{rem}[thm]{Remark}
\newtheorem*{ack}{Acknowledgement}

\newcommand{\N}{\mathbb{N}}

\newcommand{\R}{\mathbb{R}}
\newcommand{\Z}{\mathbb{Z}}
\newcommand{\sph}{\mathbb{S}}
\newcommand{\Zp}{\Z_{{(p)}}}

\newcommand{\vcd}{{\operatorname{vcd}}}
\newcommand{\rank}{{\operatorname{rank}}}

\newcommand{\symm}{\mathfrak{S}}

\newcommand{\Aut}{\operatorname{Aut}}

\begin{document}
\title[$p$-local homology of Coxeter groups]{
A vanishing theorem for \\ the $p$-local homology of Coxeter groups}
\author[T. Akita]{Toshiyuki Akita}
\address{Department of Mathematics, Hokkaido University,
Sapporo, 060-0810 Japan}
\email{akita@math.sci.hokudai.ac.jp}

\begin{abstract}
Given an odd prime number $p$ and a Coxeter group $W$ such that
the order of the product $st$ is prime to $p$ for every Coxeter generators
$s,t$ of $W$, we prove that the $p$-local homology groups $H_k(W,\Zp)$
vanish for $1\leq k\leq 2(p-2)$.
This generalize a known vanishing result for
symmetric groups due to Minoru Nakaoka.
\end{abstract}

\maketitle

\section{Introduction}
Coxeter groups are important objects in many branches of mathematics,
such as Lie theory and representation theory, 
combinatorial and geometric group theory, topology and geometry.
Since the pioneering work of Serre \cite{serre}, Coxeter groups have
been studied in group cohomology as well. See the book by Davis \cite{davis-book}
and \S\ref{sec-known} of this paper for brief outlooks.
In this paper, we will study the $p$-local homology of Coxeter groups
for \emph{odd} prime numbers $p$.
For an arbitrary Coxeter group $W$,  its integral homology group $H_k(W,\Z)$ is
known to be a finite abelian group for all $k>0$, 
and hence it decomposes into a \emph{finite} direct sum
of $p$-local homology groups each of which is a finite abelian $p$-group:
\[
H_k(W,\Z)\cong\bigoplus_p H_k(W,\Zp).
\]

According to a result of Howlett \cite{howlett},
the first and second $p$-local homology groups, $H_1(W,\Zp)$ and $H_2(W,\Zp)$,
are trivial for every odd prime number $p$. 
On the other hand, the symmetric group of $n$ letters $\symm_n$ $(n\geq 2)$
is the Coxeter group of type $A_{n-1}$.
Much is known about the (co)homology of symmetric groups.
Most notably, in his famous two papers,
Nakaoka proved the homology stability for
symmetric groups \cite{nakaoka} and computed the stable mod $p$
homology \cite{nakaoka2}. As a consequence of his results,
$H_k(\symm_n,\Zp)$ vanishes
for $1\leq k\leq 2(p-2)$ (see Theorem \ref{nakaoka-result} below).
The purpose of this paper is to generalize vanishing of $H_k(\symm_n,\Zp)$ to all
Coxeter groups:
\begin{thm}
\label{main-thm}
Let $p$ be an odd prime number
and $W$ a $p$-free Coxeter group. Then $H_k(W,\Zp)=0$ holds for $1\leq k\leq 2(p-2)$.
\end{thm}

Here a Coxeter group $W$ is said to be $p$-free if  the order of the product $st$ is prime
to $p$ for every distinct Coxeter generators $s,t$ of $W$. 
We should remark that, for $p\geq 5$, the $p$-freeness assumption is
necessary 
and the vanishing range $1\leq k\leq 2(p-2)$ is best possible. 
The situation is somewhat different for $p=3$. See \S\ref{sec-3-free}.

The proof of Theorem \ref{main-thm} consists of two steps,
a case by case argument for finite $p$-free Coxeter groups 
with relatively small rank, and the induction on the number of generators.
The induction is made possible by means of the equivariant homology of
Coxeter complexes and the Leray spectral sequence converging to the
equivariant homology.
Now we will introduce the content of this paper very briefly.
In \S\ref{subsec-defn}, we will recall definitions and
relevant facts concerning of Coxeter groups.
Known results about homology of Coxeter groups and their consequences 
will be reviewed more precisely in \S\ref{sec-known}.
After the consideration of the equivariant homology of Coxeter complexes in \S\ref{sec-equivariant}.
the proof of Theorem \ref{main-thm} will be given in \S\ref{sec-proof}. 
The final section \S\ref{sec-vanishing} consists of miscellaneous results.
There we will give some classes of Coxeter groups such that all the $p$-local homology 
groups vanish.

\begin{notation}
Throughout this paper, $p$ is an \emph{odd} prime number unless otherwise stated.
$\Zp$ is the localization of $\Z$ at the prime $p$ (the ring of $p$-local integers).
For a finite abelian group $A$, its $p$-primary component is denoted by $A_{(p)}$.
Note that $A_{(p)}\cong A\otimes_{\Z}\Zp$.
For a group $G$ and a (left) $G$-module $M$, the co-invariant of $M$ is denoted by $M_G$
(see \cite{brown}*{II.2} for the definition). For a prime number $p\geq 2$, 
we denote the cyclic group of order $p$ and the field with $p$ elements 
by the same symbol $\Z/p$.
\end{notation}

\section{Preliminaries}\label{sec-defn}
\subsection{Coxeter groups}\label{subsec-defn}
We recall definitions and relevant facts concerning of Coxeter groups.
Basic references are \cites{book-building,bourbaki,davis-book,humphreys}.
See also \cite{MR777684} for finite Coxeter groups.
Let $S$ be a finite set and $m:S\times S\rightarrow \N\cup\{\infty\}$
a map satisfying the following conditions:
\begin{enumerate}
\item $m(s,s)=1$ for all $s\in S$
\item $2\leq m(s,t)=m(t,s)\leq\infty$ for all distinct $s,t\in S$.
\end{enumerate}
The map $m$ is represented by the \emph{Coxeter graph} $\Gamma$ whose
vertex set is $S$ and whose edges are the unordered pairs $\{s,t\}\subset S$
such that $m(s,t)\geq 3$. The edges with $m(s,t)\geq 4$ are labeled
by those numbers.
The \emph{Coxeter system} associated to $\Gamma$
is the pair $(W,S)$ where $W=W(\Gamma)$ is the group generated by $s\in S$
and the fundamental relations $(st)^{m(s,t)}=1$ $(m(s,t)<\infty)$:
\[
W:=\langle s\in S\ |\ (st)^{m(s,t)}=1 (m(s,t)<\infty)\rangle.
\]
The group $W$ is called the \emph{Coxeter group} of type $\Gamma$,
and elements of $S$ are called \emph{Coxeter generators} of $W$.
The cardinality of $S$ is called the \emph{rank} of $W$
and is denoted by $|S|$ or $\rank\,W$.
Note that the order of the product $st$ is precisely $m(s,t)$.
For a subset $T\subseteq S$ (possibly $T=\varnothing$), 
the subgroup $W_T:=\langle T\rangle$ of $W$ generated
by elements $t\in T$ is called a \emph{parabolic subgroup}.
In particular, $W_S=W$ and $W_\varnothing=\{1\}$.
It is known that $(W_T,T)$ is a Coxeter system.

A Coxeter group $W$ is called \emph{irreducible} if its defining graph
$\Gamma$ is connected,
otherwise called \emph{reducible}.
For a reducible Coxeter group $W(\Gamma)$ of type 
$\Gamma$, if $\Gamma$ consists of the connected components
$\Gamma_1,\Gamma_2,\dots,\Gamma_r$, then $W(\Gamma)$
is the direct product of parabolic subgroups
$W(\Gamma_i)$'s $(1\leq i\leq r)$, each of which is irreducible:
\[
W(\Gamma)=W(\Gamma_1)\times W(\Gamma_2)\times\cdots\times W(\Gamma_r).
\]

Coxeter graphs for finite irreducible Coxeter groups are classified.
There are four infinite families $A_n\, (n\geq 1)$, $B_n\, (n\geq 2)$, $D_n\, (n\geq 4)$,
$I_2(q)\, (q\geq 3)$, and six exceptional graphs $E_6,E_7,E_8,F_4,H_3$ and $H_4$. 
The subscript indicates the rank of the resulting Coxeter group.
See Appendix for the orders of finite irreducible Coxeter groups.
Here we follow the classification given in the book by Humphreys \cite{humphreys} 
and there are overlaps $A_2=I_2(3),B_2=I_2(4)$.
Note that $W(A_n)$ is isomorphic to the symmetric group of $n+1$ letters,
while $W(I_2(q))$ is isomorphic to the dihedral group of order $2q$.

Finally,  given an odd prime number $p$,
we define a Coxeter group $W$ to be \emph{$p$-free} if $m(s,t)$ is prime to $p$
for all $s,t\in S$. Here $\infty$ is prime to all prime numbers by the convention.
For example, the Coxeter group $W(I_2(q))$ is $p$-free if and only if
$q$ is prime to $p$, while the Coxeter group $W(A_n)$ $(n\geq 2)$
is $p$-free for $p\geq 5$.
For every finite irreducible Coxeter group $W$, 
the range of odd prime numbers $p$ such that $W$ is $p$-free can be found in Appendix.
Note that parabolic subgroups of $p$-free Coxeter groups are also $p$-free. 
Henceforth, we omit the reference to the Coxeter graph $\Gamma$
and the set of Coxeter generators $S$ 
if there is no ambiguity.

\subsection{Known results for homology of Coxeter groups}\label{sec-known}
In this subsection, we will review some of known results concerning
the homology of Coxeter groups which are related to our paper. 
A basic reference for (co)homology of groups is \cite{brown}.
In the beginning, 
Serre \cite{serre} proved that every Coxeter group $W$ has
finite virtual cohomological dimension and is a group of type WFL
(see \cite{brown}*{Chapter VIII} for definitions).
This implies, in particular, that $H_k(W,\Z)$ is a
finitely generated abelian group for all $k$.
On the other hand, the rational homology of any Coxeter groups
are known to be trivial (see \cite[Proposition 5.2]{a-euler} or \cite[Theorem 15.1.1]{davis-book}).
Combining these results, we obtain the following result:

\begin{pro}\label{finite-abel}
For any Coxeter groups $W$, the integral homology group
$H_k(W,\Z)$ is a finite abelian group for all $k>0$.
\end{pro}

Consequently, 
$H_k(W,\Z)$ $(k>0)$
decomposes into a finite direct sum
\begin{equation}\label{direct-sum}
H_k(W,\Z)\cong\bigoplus_p H_k(W,\Z)_{(p)}
\end{equation}
where $p$ runs the finite set of prime numbers dividing the order of
$H_k(W,\Z)$.
The universal coefficient theorem implies 
\begin{equation}\label{univ-p-local}
H_k(W,\Z)_{(p)}\cong H_k(W,\Z)\otimes_\Z\Zp\cong H_k(W,\Zp)
\end{equation}
(see \cite[Corollary 2.3.3]{MR2378355}).
It turns out that the study of the integral homology groups of Coxeter groups reduces
to the study of the $p$-local homology groups.
Later, we will prove that $H_k(W,\Zp)=0$ $(k>0)$ if $W$ has
no $p$-torsion (Proposition \ref{p-tor-coho}).
The first and second integral homology of Coxeter groups are known.
\begin{pro}\label{low-homology}
For any Coxeter groups $W$, we have $H_1(W,\Z)\cong (\Z/2)^{n_1(W)}$
and $H_2(W,\Z)\cong (\Z/2)^{n_2(W)}$ for some non-negative integers
$n_1(W),n_2(W)$. 
\end{pro}
The claim for $H_1(W,\Z)$ is obvious because $H_1(W,\Z)=W/[W,W]$ and
$W$ is generated by elements
of order $2$. The statement for $H_2(W,\Z)$ was proved by Howlett \cite{howlett}
(following earlier works by Ihara and Yokonuma \cite{ihara-yokonuma} and
Yokonuma \cite{yokonuma}).
The nonnegative integers $n_1(W),n_2(W)$ can be computed from the Coxeter graph
for $W$.
As for $n_1(W)$, let $\mathscr{G}_W$ be the graph whose vertices
set is $S$ and whose edges are unordered pair $\{s,t\}\subset S$ such that
$m(s,t)$ is a finite odd integer. Then it is easy to see that $n_1(W)$ agrees with
the number of connected components of $\mathscr{G}_W$.
In particular, $n_1(W)\geq 1$ and hence $H_1(W,\Z)=H_1(W,\Z_{(2)})\not=0$.
For the presentation of $n_2(W)$, see \cite{howlett}*{Theorem A} or \cite[\S8.11]{humphreys}.
As a consequence of Proposition \ref{low-homology}, we obtain the following result:

\begin{cor}\label{low-p-homology} 
Let $p$ be an odd prime number.
For any Coxeter groups $W$, we have
$H_1(W,\Zp)=H_2(W,\Zp)=0$.
\end{cor}

The corollary does not hold for the third homology or higher.
Indeed, for the Coxeter group $W(I_2(q))$ of type $I_2(q)$, which is isomorphic to
the dihedral group of order $2q$ as mentioned before,
it can be proved that, if $p$ divides $q$, then
$H_k(W(I_2(q)),\Zp)\not=0$ whenever $k\equiv 3\pmod{4}$
(see \cite{MR1877725}*{Theorem 2.1} and \S\ref{sec:aspherical} below).
This observation also shows the necessity of the $p$-freeness assumption in our results
for $p\geq 5$.
Finally, we will recall a consequence of results of Nakaoka 
\cites{nakaoka,nakaoka2} which was mentioned in the introduction.

\begin{thm}[Nakaoka \cites{nakaoka,nakaoka2}]\label{nakaoka-result}
Let $\symm_n$ be the symmetric group of $n$ letters. Then
$H_k(\symm_n,\Zp)=0$ $(1\leq k\leq 2(p-2))$ for all odd prime numbers $p$.
\end{thm}
\begin{proof}
In his paper \cite{nakaoka}, Nakaoka proved the homology stability for
symmetric groups.
Namely, for $2\leq m\leq n\leq\infty$, the homomorphism 
$H_k(\symm_m,A)\rightarrow H_k(\symm_n,A)$ induced by
the natural inclusion $\symm_m\hookrightarrow\symm_n$ is injective for all $k$, 
and is an isomorphism if $k<(m+1)/2$, where $A$ is an abelian group with the trivial
$\symm_n$-action, and $\symm_\infty$ is the infinite symmetric group 
\cite[Theorem 5.8 and Corollary 6.7]{nakaoka}.
He also computed the mod $p$ homology of $\symm_\infty$ in
\cite[Theorem 7.1]{nakaoka2}, from which we deduce
that $H_k(\symm_\infty,\Z/p)=0$ for 
$1\leq k\leq 2(p-2)$ and that $H_{2p-3}(\symm_\infty,\Z/p)\not=0$.
Combining these results, we see that $H_k(\symm_n,\Z/p)=0$
$(1\leq k\leq 2(p-2))$ for all $n$.
Applying the universal coefficient theorem, the theorem follows.
\end{proof}

Theorem \ref{main-thm}, together with Corollary \ref{low-p-homology} for $p=3$,
generalize Theorem \ref{nakaoka-result} to all Coxeter groups.
For further results concerning of (co)homology of Coxeter groups,
we refer the book by Davis \cite{davis-book} and
papers \cites{MR1334713,MR1603123,akita-quillen,MR1764318,MR1446573,pride-stohr,MR1877725,MR2709085} 
as well as references therein.

\section{Coxeter complexes and their equivariant homology}\label{sec-equivariant}
\subsection{Coxeter complexes}\label{complex}
We recall the definition and properties of Coxeter complexes
which are relevant to prove Theorem \ref{main-thm}.
A basic reference for Coxeter complexes is \cite{book-building}*{Chapter 3}.
Given a Coxeter group $W$, the \emph{Coxeter complex} 
$X_W$ of $W$ is the poset of
cosets $wW_T$ $(w\in W, T\subsetneq S)$,
ordered by reverse inclusion. 
It is known that $X_W$ is actually an $(|S|-1)$-dimensional
simplicial complex (see \cite[Theorem 3.5]{book-building}). 
The $k$-simplices of $X_W$ are the cosets $wW_T$ with $k=|S|-|T|-1$.
A coset $wW_T$ is a face of $w'W_{T'}$ if and only if $wW_T\supseteq w'W_{T'}$.
In particular, the vertices are cosets of the form $wW_{S\setminus\{s\}}$
$(s\in S,w\in W)$, 
the maximal simplices are the singletons $wW_\varnothing=\{w\}$ $(w\in W)$,
and the codimension one simplices are cosets of the form
$wW_{\{s\}}=\{w,ws\}$ $(s\in S, w\in W)$.
In what follows, we will not distinguish between $X_W$ and its
geometric realization.

There is a simplicial action of $W$ on $X_W$ by left translation 
$w'\cdot wW_T:=w'wW_T$. The isotropy subgroup of a simplex $wW_T$ is
precisely $wW_T w^{-1}$, which
fixes $wW_T$ pointwise.
Next, consider the subcomplex $\Delta_W=\{W_T\ |\ T\subsetneq S\}$ of $X_W$,
which consists of a single $(|S|-1)$-simplex $W_\varnothing$ and its faces.
Since the \emph{type function}  $X_W\rightarrow S$,
$wW_T\mapsto S\setminus T$ is well-defined
(see \cite[Definition 3.6]{book-building}), $\Delta_W$ forms the
set of representatives of $W$-orbits of simplices of $X_W$.
The following fact is well-known.
\begin{pro}
If $W$ is a finite Coxeter group, then $X_W$ is a triangulation of the
$(|S|-1)$-dimensional sphere $\sph^{|S|-1}$.
\end{pro}
See \cite[Proposition 1.108]{book-building} for the proof.
Alternatively, $W$ can be realized
as an orthogonal reflection group on the $|S|$-dimensional Euclidean space
$\R^{|S|}$ and hence it acts on the unit sphere $\sph^{|S|-1}$. 
Each $s\in S$ acts on $\sph^{|S|-1}$
as an orthogonal reflection.
The Coxeter complex $X_W$ coincides
with the equivariant triangulation of $\sph^{|S|-1}$ cut out by the
reflection hyperplanes for $W$.
 In case $W$ is infinite, Serre proved the following result:
\begin{pro}[{\cite[Lemma 4]{serre}}]
\label{serre-cont} If $W$ is an infinite Coxeter group, 
then $X_W$ is contractible.
\end{pro}

\subsection{Equivariant homology of Coxeter complexes}
Given a Coxeter group $W$, 
let $H_k^W(X_W,\Zp)$ be the $k$-th equivariant homology group of $X_W$
(see \cite[Chapter VII]{brown} for the definition).
If $X_W$ is infinite, then $X_W$ is contractible so that the equivariant homology is
isomorphic to the homology of $W$:
\begin{pro}\label{compare0}
If $W$ is an infinite Coxeter group, then 
\[H_k^W(X_W,\Zp)\cong H_k(W,\Zp)\]
for all $k$. 
\end{pro}
If $W$ is finite, then $H_k^W(X_W,\Zp)$ may not be isomorphic to $H_k(W,\Zp)$,
however, they are isomorphic if $k$ is relatively small:
\begin{pro}\label{compare}
If $W$ is a finite Coxeter group, then 
\[
H_k^W(X_W,\Zp)\cong H_k(W,\Zp)
\]
for  $k\leq \rank\,W-1$.
\end{pro}
\begin{proof}
Consider the spectral sequence
\[
E_{i,j}^2=H_i(W,H_j(X_W,\Zp))\Rightarrow H_{i+j}^W(X_W,\Zp)
\]
(see \cite[VII.7]{brown})
and note that $E_{i,0}^2\cong H_i(W,\Zp)$ for all $i$.
Since $X_W$ is homeomorphic to $\sph^{|S|-1}$, we have $E_{i,j}^2=0$ for
$j\not=0,|S|-1$. Hence $H_k^W(X_W,\Zp)\cong H_k(W,\Zp)$ for
$k\leq |S|-2$. Now
\[
E^2_{0,|S|-1}=H_0(W,H_{|S|-1}(X_W,\Zp))= H_{|S|-1}(X_W,\Zp)_W
\]
where the RHS is the co-invariant of $H_{|S|-1}(X_W,\Zp)$ as a $W$-module
(see \cite[III.1]{brown}).
Since each $s\in S$ acts on $X_W\approx\sph^{|S|-1}$ as an
orthogonal reflection as mentioned in \S\ref{complex}, 
it acts on $H_{|S|-1}(X_W,\Zp)\cong\Zp$ as the multiplication
by $-1$. It follows that the co-invariant $H_{|S|-1}(X_W,\Zp)_W$ is isomorphic to the quotient
group of $\Zp$ by the subgroup generated by $r-(-1)r=2r$ $(r\in\Zp)$.
But this subgroup is nothing but the whole group $\Zp$ because $2$ is invertible in $\Zp$. 
This proves $E^2_{0,|S|-1}=0$ and hence
\[
H_{|S|-1}^W(X_W,\Zp)\cong H_{|S|-1}(W,\Zp)
\]
as desired.
\end{proof}

\section{Proof of Theorem \ref{main-thm}}\label{sec-proof}
We will prove Theorem \ref{main-thm} by showing the following two claims:
\begin{claim}\label{claim1}
If $W$ is a finite $p$-free Coxeter group with $\rank\,W\leq 2(p-2)$,
then $H_k(W,\Zp)=0$ for $1\leq k\leq 2(p-2)$.
\end{claim}
\begin{claim}\label{claim2}
Claim \ref{claim1} implies Theorem \ref{main-thm}.
\end{claim}
The first claim is equivalent to Theorem \ref{main-thm} for finite $p$-free
Coxeter groups with $\rank\,W\leq 2(p-2)$, and will be proved by a case by case argument.
The second claim will be proved by the induction on $\rank\,W$ by
using the equivariant homology of Coxeter complexes.
Let us prove Claim \ref{claim2} first.

\subsection{Proof of Claim \ref{claim2}}
For every Coxeter group $W$,
there is a spectral sequence
\begin{equation}\label{leray-ori}
E^1_{i,j}=\bigoplus_{\sigma\in\mathscr{S}_i}
H_j(W_\sigma,\Zp)\Rightarrow H_{i+j}^W(X_W,\Zp),
\end{equation}
where $\mathscr{S}_i$ is the set of representatives of $W$-orbits of $i$-simplices of $X_W$,
and $W_\sigma$ is the isotropy subgroup of an $i$-simplex $\sigma$
(see \cite[VII.7]{brown}).
It is the Leray spectral sequence for the natural projection
$EW\times_W X_W\rightarrow X_W/W$.
Note that $\Zp$ in $H_j(W_\sigma,\Zp)$ is the trivial $W_\sigma$-module
because $W_\sigma$ fixes $\sigma$ pointwise.
We may choose the subset
$\{W_T\ |\ T\subsetneq S, |T|=|S|-i-1\}$
(the set of $i$-simplices of $\Delta_W$) as $\mathscr{S}_i$,
and the spectral sequence can be rewritten as
\begin{equation}\label{leray}
E^1_{i,j}=\bigoplus_{\substack{T\subsetneq S \\ |T|=|S|-i-1}}
H_j(W_T,\Zp)\Rightarrow H_{i+j}^W(X_W,\Zp).
\end{equation}
\begin{lem}\label{base-term}
In the spectral sequence (\ref{leray}), $E^2_{i,0}=0$ for $i\not=0$ and
$E^2_{0,0}\cong\Zp$.
\end{lem}
\begin{proof}
We claim $E^2_{i,0}\cong H_i(\Delta_W,\Zp)$ for all $i$, which implies
the lemma because $\Delta_W$ is an $(|S|-1)$-simplex and hence contractible.
Although such a claim may be familiar to experts, we write down the proof
for completeness.
To show the claim, we recall the construction of the spectral sequence
(\ref{leray-ori}) given in \cite{brown}*{VII.7}. 
At the first stage, the $E^1_{i,0}$-term of (\ref{leray-ori}) is given by
\[
E^1_{i,0}=H_0(W,C_i(X_W,\Zp))=C_i(X_W,\Zp)_W,
\]
which is isomorphic to the one in (\ref{leray-ori}) due to Eckmann-Shapiro lemma.
The differential $d^1:E^1_{i,0}\rightarrow E^1_{i-1,0}$
is the map induced by the boundary operator
$C_i(X_W,\Zp)\rightarrow C_{i-1}(X_W,\Zp)$.
On the other hand, the composition
\begin{equation}\label{fundom}
C_i(\Delta_W,\Zp)\hookrightarrow C_i(X_W,\Zp)
\twoheadrightarrow C_i(X_W,\Zp)_W
\end{equation}
is an isomorphism, where the first map is induced by the inclusion
$\Delta_W\hookrightarrow X_W$ and the second map is the natural
projection, because the subcomplex $\Delta_W$ forms the
set of representatives of $W$-orbits of simplices of $X_W$.
Moreover, the isomorphism (\ref{fundom}) is compatible
with the boundary operator of $C_i(\Delta_W,\Zp)$ and
the differential on $C_i(X_W,\Zp)_W$.
In other words, (\ref{fundom}) yields a chain isomorphism
of chain complexes
\[
(C_i(\Delta_W,\Zp),\partial)\rightarrow (C_i(X_W,\Zp)_W,d^1).
\]
The claim follows immediately.
\end{proof}

\begin{proof}[Proof of Claim \ref{claim2}]
We argue by the induction on $|S|$. 
When $W$ is finite, we may assume $|S|> 2(p-2)$, for we suppose
that Claim \ref{claim1} holds.
Consider the spectral sequence (\ref{leray}).
Observe first that all $W_T$'s appearing in (\ref{leray}) are $p$-free and satisfy
$|T|<|S|$.
By the induction assumption, we have $H_j(W_T,\Zp)=0$
$(1\leq j\leq 2(p-2))$ for 
all $T\subsetneq S$, which implies $E_{i,j}^1=E_{i,j}^2=0$ for $1\leq j\leq 2(p-2)$.
Moreover, $E^2_{i,0}=0$ for $i>0$ by Lemma \ref{base-term}. This proves
$H_k^W(X_W,\Zp)=0$ for $1\leq k\leq 2(p-2)$. Now Claim \ref{claim2} follows from
Proposition \ref{compare0} and \ref{compare}.
\end{proof}

\subsection{Proof of Claim \ref{claim1}}
Given an odd prime $p$, 
if $W$ is a finite $p$-free Coxeter group with $\rank\,W\leq 2(p-2)$,
then $W$ decomposes into the direct product of
finite irreducible $p$-free Coxeter groups
$W\cong W_1\times\cdots\times W_r$
with $\Sigma_{i=1}^r\rank\,W_i=\rank\,W$.
Since $\Zp$ is PID, we may apply the K\"unneth theorem to conclude that
Claim \ref{claim1} is equivalent to the following claim:
\begin{claim}\label{claim3}
If $W$ is a finite irreducible $p$-free Coxeter group with $\rank\,W\leq 2(p-2)$,
then $H_k(W,\Zp)$ vanishes for $1\leq k\leq 2(p-2)$.
\end{claim}

We prove Claim \ref{claim3} for each finite irreducible Coxeter group.
Firstly, the Coxeter group $W(I_2(q))$ of type $I_2(q)$ is $p$-free if and only if
$q$ is prime to $p$. If so, $H_*(W(I_2(q)),\Zp)=0$ for $*>0$ because the order of
$W(I_2(q))$ is $2q$ and hence having no $p$-torsion. 
Next, we prove the claim for the Coxeter group of type $A_n$.
To do so, we deal with cohomology instead of homology.
We invoke the following elementary lemma:
\begin{lem}\label{duality}
 Let $G$ be a finite group and $p\geq 2$ a prime. Then
$H_k(G,\Zp)\cong H^{k+1}(G,\Z)_{(p)}$ for all $k\geq 1$. 
\end{lem}

Now Claim \ref{claim3} for $W(A_n)$ can be proved by applying standard arguments
in cohomology of finite groups:

\begin{lem}\label{case-symm}
$H_k(W(A_n),\Zp)=0$ $(1\leq k \leq 2(p-2))$ holds for all $n$ with 
$1\leq n \leq 2(p-1)$. 
\end{lem}
\begin{proof}
Recall that $W(A_n)$ is isomorphic to the symmetric group of 
$n+1$ letters $\symm_{n+1}$.
If $n<p$ then $\symm_n$ has no $p$-torsion and hence
$H_k(\symm_n,\Zp)=0$ for all $k>0$. 
Now suppose $p\leq n\leq 2p-1$, and let $C_p$ be a Sylow $p$-subgroup
of $\symm_n$, which is a cyclic group of order $p$. Then $H^*(C_p,\Z)\cong
\Z[u]/(pu)$ where $\deg u=2$.
Let $N_p$ be the normalizer 
of $C_p$ in $\symm_n$.
It acts on $C_p$ by conjugation, and the induced map
$N_p\rightarrow\Aut (C_p)\cong (\Z/p)^\times$ is known to be surjective.
Consequently, the invariant $H^*(C_p,\Z)^{N_p}$ is the subring generated by $u^{p-1}$.
Since $C_p$ is abelian, the restriction $H^*(\symm_n,\Z)\rightarrow
H^*(C_p,\Z)$ induces the isomorphism
\[
H^k(\symm_n,\Z)_{(p)}\cong H^k(C_p,\Z)^{N_p}
\]
for $k>0$ by a result
of Swan \cite{swan}*{Lemma 1 and Appendix} (see also \cite{thomas-book}*{Lemma 3.4}).
This proves, for $p\leq n\leq 2p-1$, that $H^k(\symm_n,\Z)_{(p)}=0$ $(0<k<2p-2)$ and
$H^{2p-2}(\symm_n,\Z)_{(p)}\not=0$.
In view of Lemma \ref{duality}, the proposition follows.
\end{proof}

\begin{rem}\label{best-possible}
Since $H_{2p-3}(W(A_{p-1}),\Zp)\cong H^{2p-2}(\symm_p,\Z)_{(p)}\not=0$
for all prime numbers $p$ as was observed in the proof of Lemma \ref{case-symm},
the vanishing range $1\leq k\leq 2(p-2)$ in our theorem is best possible for $p\geq 5$.
\end{rem}

\begin{rem} 
Of course, Lemma \ref{case-symm} is a direct consequence of
Theorem \ref{nakaoka-result}, however, we avoid the use of Theorem \ref{nakaoka-result}
for two reasons:
Firstly, by doing so, we provide an alternative proof of Theorem \ref{nakaoka-result}.
Secondly, the proof of Lemma \ref{case-symm} is much simpler than that of 
Theorem \ref{nakaoka-result}, for the latter relies on the homology stability for
symmetric groups and the computation of $H_*(\symm_\infty,\Z/p)$.
\end{rem}

Claim \ref{claim3} for the Coxeter groups of type $B_n$ and $D_n$ follows from
Lemma \ref{case-symm} and the following proposition:
\begin{pro}\label{ABD} For any odd prime number $p$,
\[
H_*(W(B_n),\Zp)\cong H_*(W(A_{n-1}),\Zp)
\]
holds for all $n\geq 2$, and
\[
H_*(W(D_n),\Zp)\cong H_*(W(A_{n-1}),\Zp)
\]
holds for all $n\geq 4$.
\end{pro}
\begin{proof}
Recall that the Coxeter group $W(B_n)$ is isomorphic to the semi-direct product
$(\Z/2)^n \rtimes W(A_{n-1})$ (see \cite{davis-book}*{\S6.7} or \cite[\S1.1]{humphreys}).
In the Lyndon-Hochschild-Serre spectral sequence
\[
E^2_{i,j}=H_i(W(A_{n-1}),H_j((\Z/2)^n,\Zp))\Rightarrow H_{i+j}(W(B_n),\Zp),
\]
one has $E^2_{i,j}=0$ for $j\not=0$ since $H_j((\Z/2)^n,\Zp)=0$ for $j\not=0$.
This proves $H_*(W(B_n),\Zp)\cong H_*(W(A_{n-1}),\Zp)$.
On the other hand,
$W(D_n)$ is known to be isomorphic to the semi-direct product
$(\Z/2)^{n-1} \rtimes W(A_{n-1})$ (see loc.~cit.), and the proof for
$H_*(W(D_n),\Zp)\cong H_*(W(A_{n-1}),\Zp)$ is similar.
\end{proof}
These observations prove Claim \ref{claim3} for $p\geq 11$,
for all finite irreducible Coxeter groups of type other than $A_n,B_n,D_n$ and $I_2(q)$
have no $p$-torsion for $p\geq 11$.
The case $p=3$ follows from Corollary \ref{low-p-homology}.
Now we will prove the cases $p=5$ and $p=7$.
Observe that, apart from Coxeter groups of type $A_n$, $B_n$, $D_n$ and $I_2(q)$,
finite irreducible $p$-free Coxeter groups, with rank at most $2(p-2)$ and having $p$-torsion, are
$W(E_6)$ for $p=5$, $W(E_7)$ and $W(E_8)$ for $p=7$.
So the proof of Claim \ref{claim3} is completed by showing the following lemma:

\begin{lem}
$H_k(W(E_6),\Z_{(5)})$ vanishes for $1\leq k\leq 6$, while 
$H_k(W(E_7),\Z_{(7)})$ and $H_k(W(E_8),\Z_{(7)})$ vanish for $1\leq k\leq 10$.
\end{lem}
\begin{proof}
The Coxeter group $W(A_4)$ is a parabolic subgroup of $W(E_6)$, and they have a common
Sylow $5$-subgroup $C_5$, which is a cyclic group of order $5$. 
The transfer homomorphism to the Sylow $5$-subgroup
$H_k(W(E_6),\Z_{(5)})\rightarrow H_k(C_5,\Z_{(5)})$ is injective and factors 
into a composition
of transfer homomorphisms
\[
H_k(W(E_6),\Z_{(5)})\rightarrow H_k(W(A_4),\Z_{(5)})\rightarrow H_k(C_5,\Z_{(5)}).
\]
In view of Lemma \ref{case-symm}, we conclude that $H_k(W(E_6),\Z_{(5)})=0$ for
$1\leq k\leq 6$, which proves the lemma for $W(E_6)$.
On the other hand, there is a sequence of parabolic subgroups $W(A_6)<W(E_7)<W(E_8)$,
and they have a common Sylow $7$-subgroup $C_7$, which is a cyclic group of order $7$.
The proof of the lemma for $W(E_7)$ and $W(E_8)$ is similar.
\end{proof}

\section{Coxeter groups with vanishing $p$-local homology}\label{sec-vanishing}
In this final section, we introduce some families of Coxeter groups such that
$H_k(W,\Zp)$ vanishes for all $k>0$.

\subsection{Aspherical Coxeter groups}\label{sec:aspherical}
A Coxeter group $W$ is called \emph{aspherical} in \cite{pride-stohr} if,
for all distinct Coxeter generators $s,t,u\in S$, the inequality
\[
\frac{1}{m(s,t)}+\frac{1}{m(t,u)}+\frac{1}{m(u,s)}\leq 1
\]
holds, where $1/\infty=0$ by the convention.
The inequality is equivalent to the condition that the parabolic subgroup
$W_{\{s,t,u\}}$ is of infinite order. The (co)homology groups of aspherical Coxeter
groups were studied by Pride and St{\"o}hr \cite{pride-stohr},
and the mod $2$ cohomology rings of aspherical Coxeter groups were
studied by the author \cite{akita-quillen}.
Among other things,
Pride and St{\"o}hr obtained the following exact sequence
\[
\cdots\rightarrow H_{k+1}(W,A)\rightarrow \bigoplus_{s\in S}H_k(W_{\{s\}},A)^{
\oplus n(s)}
\rightarrow\!\bigoplus_{\substack{\{s,t\}\subset S \\ m(s,t)<\infty}}\!
H_k(W_{\{s,t\}},A)\rightarrow H_k(W,A)\rightarrow\cdots
\]
terminating at $H_2(W,A)$, where $A$ is a $W$-module and $n(s)$ is a certain
nonnegative integer defined for each $s\in S$
\cite[Theorem 5]{pride-stohr}.
Since $W_{\{s\}}\cong\Z/2$, $H_k(W_{\{s\}},\Zp)=0$ for $k>0$.
Moreover, if $p$ does not divide $m(s,t)$, then $H_k(W_{\{s,t\}},\Zp)=0$ for $k>0$ either.
Here no prime numbers $p$ divide $\infty$ by the convention.
Hence we obtain the following result
(the statement for $k=1,2$ follows from Corollary  \ref{low-p-homology}):
\begin{pro}
For any aspherical Coxeter groups $W$, we have
\[
H_k(W,\Zp)\cong 
\bigoplus_{\substack{\{s,t\}\subset S \\ p\,|\,m(s,t)}}
H_k(W_{\{s,t\}},\Zp)
\]
for all $k>0$. Furthermore, $H_k(W,\Zp)$ vanishes for all $k>0$
if and only if $W$ is $p$-free.
\end{pro}
Note that if $p$ divides $m(s,t)$, then
\[
H_k(W_{\{s,t\}},\Zp)\cong
\begin{cases}(\Z/m(s,t))_{(p)} & k\equiv 3\pmod{4} \\
0 & k\not\equiv 3\pmod{4}
\end{cases}
\]
for $k>0$, where $\Z/m(s,t)$ is the cyclic group of order $m(s,t)$
(see \cite{MR1877725}*{Theorem 2.1}).
\subsection{Coxeter groups without $p$-torsion}
Next we prove vanishing of the $p$-local homology of Coxeter groups without $p$-torsion.
Before doing so, we characterize such Coxeter groups in terms of their finite
parabolic subgroups.
\begin{pro}\label{p-tor}
Let $p$ be a prime number.
A Coxeter group $W$ has no $p$-torsion
if and only if every finite parabolic subgroup
has no $p$-torsion.
\end{pro} 
\begin{proof}
According to a result of Tits, 
every finite subgroup of $W$ is contained in 
conjugate of some parabolic subgroup of finite order
(see \cite[Corollary D.2.9]{davis-book}).
The proposition follows at once.
\end{proof}
\begin{pro}\label{p-tor-coho}
If $W$ is a Coxeter group without $p$-torsion, then
$H_k(W,\Zp)=0$ for all $k>0$.
\end{pro}
\begin{proof}
The claim is obvious for finite Coxeter groups.
We prove the proposition for infinite Coxeter groups by the induction on $|S|$.
Let $W$ be an infinite Coxeter group without $p$-torsion and consider the
spectral sequence (\ref{leray}).
Every proper parabolic subgroup $W_T$ of $W$ has no $p$-torsion,
and hence $H_k(W_T,\Zp)=0$ $(k>0)$ by the induction assumption.
This implies $E^1_{i,j}=0$ for $j\not=0$. Moreover
$E^2_{i,j}=0$ for $(i,j)\not=(0,0)$ by Lemma \ref{base-term},
which proves the proposition.
\end{proof}
In view of the last proposition, the direct sum decomposition
(\ref{direct-sum}) can be replaced to the following:
\begin{cor}\label{cor-decomp} 
For any Coxeter groups $W$ and $k>0$, we have
\[
H_k(W,\Z)\cong\bigoplus_p H_k(W,\Zp),
\]
where $p$ runs prime numbers such that $W$ has $p$-torsion.
\end{cor}
\begin{rem}
Proposition \ref{p-tor-coho} and Corollary \ref{cor-decomp}
should be compared with the following general results.
Namely, suppose that $\Gamma$ is a group having 
finite virtual cohomological dimension $\vcd\,\Gamma$.
If $\Gamma$ does not have $p$-torsion, then $H^k(\Gamma,\Z)_{(p)}=0$
for $k>\vcd\,\Gamma$.
Consequently, we have the finite direct product decomposition
\[
H^k(\Gamma,\Z)\cong\prod_p H^k(\Gamma,\Z)_{(p)}
\]
which holds for $k>\vcd\,\Gamma$, 
where $p$ ranges over the prime numbers such that
$\Gamma$ has $p$-torsion. See \cite[Chapter X]{brown}.
\end{rem}

\subsection{Right-angled Coxeter groups}
A Coxeter group is called \emph{right-angled} if $m(s,t)=2$ or $\infty$
for all distinct $s,t\in S$. The mod $2$ cohomology rings of
right-angled Coxeter groups were determined by Rusin \cite{rusin-thesis} 
(see also \cite[Theorem 15.1.4]{davis-book}).
In this section, we prove vanishing of $p$-local homology for a class of Coxeter groups
which includes right-angled Coxeter groups.

\begin{pro}
If\, $W$ is a Coxeter group such that $m(s,t)$ equals to the power of 2 or $\infty$
for all distinct $s,t\in S$, 
then $H_k(W,\Zp)=0$ $(k>0)$ for all odd prime numbers $p\geq 3$.
\end{pro}
\begin{proof}
The finite irreducible Coxeter groups satisfying the assumption 
are $W(A_1)$ (of order $2$), $W(B_2)$ (of order $8$), and 
$W(I_2(2^m))$'s (of order $2^{m+1}$).
Every finite parabolic subgroup of $W$ is isomorphic to a direct product
of copies of those groups and hence has the order the power of $2$.
Consequently, $W$ has no $p$-torsion by Proposition \ref{p-tor}.
Now the proposition follows from Proposition \ref{p-tor-coho}.
\end{proof}

\subsection{$3$-free Coxeter groups}\label{sec-3-free}
In this final section, we look into situations for $p=3$ more closely.
Firstly, according to Corollary \ref{low-p-homology}, 
$H_1(W,\Z_{(3)})=H_2(W,\Z_{(3)})=0$ for any Coxeter groups $W$.
This means that Theorem \ref{main-thm} remains true for $p=3$ without
$3$-freeness assumption.
On the other hand, the finite irreducible $3$-free Coxeter groups are
$W(A_1),W(B_2)$ and $W(I_2(q))$ such that $q$ is prime to $3$, all of which
have no $3$-torsion. Consequently, every $3$-free Coxeter group has
no $3$-torsion by Proposition \ref{p-tor}. Applying Proposition \ref{p-tor-coho}
we obtain the following result:
\begin{pro}\label{pro-3-free} 
For every $3$-free Coxeter group,
$H_k(W,\Z_{(3)})=0$ holds for all $k>0$.
\end{pro}

\section*{Appendix}
The following is the table for the Coxeter graph $\Gamma$, 
the order $|W(\Gamma)|$ of the corresponding 
Coxeter group $W(\Gamma)$, the order $|W(\Gamma)|$ factored into primes, and
the range of odd prime numbers $p$ such that $W(\Gamma)$ is $p$-free.
\[
\begin{array}{cccc}
\Gamma & |W(\Gamma)| && \text{$p$-freeness} \\ \hline
A_1 & 2 && p\geq 3 \\
A_n\,(n\geq 2) & (n+1)! && p\geq 5 \\
B_2 & 8 && p\geq 3\\
B_n\,(n\geq 3) & 2^n n! && p\geq 5\\
D_n\,(n\geq 4) & 2^{n-1} n! && p\geq 5\\
E_6 & 72\cdot 6! & 2^7\cdot 3^4\cdot 5 & p\geq 5 \\
E_7 & 72\cdot 8! & 2^{10}\cdot 3^4\cdot 5\cdot 7 & p\geq 5 \\
E_8 & 192\cdot 10! & 2^{14}\cdot 3^5\cdot 5^2\cdot 7 & p\geq 5 \\
F_4 & 1152 & 2^7\cdot 3^2 & p\geq 5 \\
H_3 & 120 & 2^3\cdot 3\cdot 5 & p\geq 7 \\
H_4 & 14400 & 2^6\cdot 3^2\cdot 5^2 & p\geq 7 \\
I_2(q)\,(q\geq 3) & 2q && p\not|\ q
\end{array}
\]

\begin{ack}
This study started from questions posed by Takefumi Nosaka
concerning of the third $p$-local homology of Coxeter groups for
odd prime numbers $p$.
The author thanks to him for drawing our attention.
This study was partially supported by JSPS KAKENHI Grant Numbers
23654018 and 26400077.
\end{ack}

\end{document}